\documentclass[11pt,a4paper,leqno]{amsart}
\usepackage{amssymb,xspace} 
\usepackage{amstext}
\usepackage[mathscr]{eucal} 
\theoremstyle{plain}
\usepackage{amsbsy,amssymb,amsfonts,latexsym,eucal,amscd}
\usepackage[dvips]{graphicx}
\usepackage{epsfig}
\usepackage[all]{xy}
\usepackage{pstricks}
\usepackage{mathrsfs}
\newcommand{\tens}[1][]{\mathbin{\otimes_{\raise1.5ex\hbox to-.1em{}{#1}}}}
\newcommand{\lltens}[1][]{{\mathop{\tens}\limits^{\rm \mathbb{L}}}_{#1}}

\DeclareMathOperator{\ch}{ch}

\DeclareMathOperator{\supp}{supp}
\DeclareMathOperator{\td}{td}

\DeclareMathOperator{\pr}{pr}

\newcommand{\oo }{\ensuremath{\mathcal{O}}}

\newcommand{\ff }{\ensuremath{\mathcal{F}}}

\newcommand{\g }{\ensuremath{\mathcal{G}}}

\newcommand{\tim }{\ensuremath{\times}}

\newtheorem{theorem}{Theorem}[section]

\newtheorem{lemma}[theorem]{Lemma}

\newtheorem{proposition}[theorem]{Proposition}

{\theoremstyle{definition}}

{\theoremstyle{definition}}

{\theoremstyle{definition}}

{\theoremstyle{definition}}

{\theoremstyle{definition}\newtheorem{definition}[theorem]{Definition}}

{\theoremstyle{definition}}

{\theoremstyle{definition}}

\address{CMI, UMR CNRS 6632 (LATP) \\
Universit\'e de Provence\\
39, rue Fr\'ed\'eric Joliot-Curie\\
13453 Marseille Cedex 13\\France\\}
\email{jgrivaux@cmi.univ-mrs.fr}

\newcommand{\he}{^{\vphantom{*}} }
\newcommand{\be}{_{\vphantom{i}} }

\newcommand{\ti }[1]{\ensuremath{\widetilde{#1}}}

\newcommand{\ee }{\ensuremath{^{\, *}}}

\newcommand{\bop }{\ensuremath{\bigoplus\limits}}

\newcommand{\pe }{\ensuremath{^{\, !}}}

\def\P{\ensuremath{\mathbb P}}

\entrymodifiers={+!!<0pt,\fontdimen22\textfont2>}

\def\apl#1#2#3{#1\mkern -4 mu:\mkern - 8 mu
\xymatrix@C=17pt{#2\!\ar[r]&\!#3}
}

\def\apllongue#1#2#3{#1\mkern -4 mu:\mkern - 8 mu
\xymatrix{#2\!\ar[r]&\!#3}
}

\def\fl{\xymatrix@C=17pt{
\ar[r]&
}}

\def\flcourte{\xymatrix@C=12pt{
\ar[r]&
}}

\def\flgd#1#2{\xymatrix@C=17pt{#1\!
\ar[r]&\!#2
}}

\DeclareMathOperator{\id}{id}

\author{Julien Grivaux}
\title[On a conjecture of Kashiwara]{On a conjecture of Kashiwara relating Chern and Euler classes of $\oo$-modules}

\begin{document}
\begin{abstract}
 In this note we prove a conjecture of Kashiwara, which states that the Euler class of a coherent analytic sheaf $\ff$ on a complex manifold $X$ is the product of the Chern character of $\ff$ with the Todd class of $X$. As a corollary, we obtain a functorial proof of the Grothendieck-Riemann-Roch theorem in Hodge cohomology for complex manifolds.
\end{abstract}
\maketitle
\section{Introduction}
\setcounter{footnote}{0}
The notation used throughout this article is defined in \S \ref{SecDeux}.
\par\medskip
Let $X$ be a complex manifold, $\omega_{X}\he$ be the holomorphic dualizing complex of $X$, $\delta _{X}\he$ be the diagonal injection of $X$ in $X \times X$ and $\textrm{D}_{\textrm{coh}}^{\textrm{b}}(X)$ be the full subcategory of the bounded derived category of analytic sheaves on $X$ consisting of objects with coherent cohomology. In the letter \cite{Ka} which is reproduced in Chapter 5 of \cite{KaSc}, Kashiwara constructs for every $\ff$ in $\textrm{D}_{\textrm{coh}}^{\textrm{b}}(X)$ two cohomology classes
$\textrm{hh}_{X}\he(\ff)$ and $\textrm{thh}_{X}\he(\ff)$ in $\textrm{H}_{\supp(\ff)}^{0}(X,\delta _{X}^{*}\delta _{X*}\he\oo_{X}\he)$ and $\textrm{H}_{\supp(\ff)}^{0}(X,\delta _{X}^{!}\delta _{X!}\he\,\omega _{X}\he)$; they are the Hochs\-child and co-Hochschild classes of $\ff$.
\par \medskip
Let us point out that characteristic classes in Hoch\-schild homology are well-known in homological algebra (see \cite[\S 8]{Loday}). They have been recently intensively studied in various algebraico-geometric contexts. For further details, we refer the reader to \cite{Ca1}, \cite{Ca2}, \cite{Sh} and to the references therein.
\par \medskip
If $\apl{f}{X}{Y}$ is a holomorphic map, the classes $\textrm{hh}_{X}\he$ and $\textrm{thh}_{X}\he$ satisfy the following dual functoriality properties:
\par\smallskip
\hspace{0.4 cm}-- for every $\g$ in $\textrm{D}_{\textrm{coh}}^{\textrm{b}}(Y)$,
$\textrm{hh}_{X}\he(f^{*}\be\g)=f^{*}\be\,\textrm{hh}_{Y}\he(\g)$,
\par \smallskip
\hspace{0.4 cm}-- for every $\ff$ in $\textrm{D}_{\textrm{coh}}^{\textrm{b}}(X)$ such that
$f$ is proper on $\supp(\ff)$,
\[
\textrm{thh}_{Y}\he(Rf_{!}\he\,\ff)=f_{!}\he\,\textrm{thh}_{X}\he(\ff).
\]
\par \medskip
The analytic Hochschild-Kostant-Rosenberg isomorphisms construc\-ted in \cite{Ka} are specific isomorphisms
\[
 \delta _{X}^{*}\delta_{X*}\he\oo_{X}\he\simeq\xymatrix{\bop\limits_{i \geq 0}}\Omega_{X}^{\, i}[i]\qquad \textrm{and}\qquad\delta _{X}^{!}\delta_{X!}\he\,\omega _{X}\he\simeq\xymatrix{\bop\limits_{i \geq 0}} \Omega_{X}^{\, i}[i]
\]
in $\textrm{D}_{\textrm{coh}}^{\textrm{b}}(X)$. The Hochschild and co-Hochschild classes of an element $\ff$ in $\textrm{D}_{\textrm{coh}}^{\textrm{b}}(X)$ are mapped via the above HKR isomorphisms to the so-called Chern and Euler classes of $\ff$ in $\smash[b]{\bop_{i \geq 0} \textrm{H}_{\supp(\ff)}^{i}(X,\Omega _{X}^{\,i})}$, denoted by $\smash{\textrm{ch}(\ff^{\vphantom{}})}$ and $\smash{\textrm{eu}(\ff)}$.
\par
\medskip
The natural morphism
\[
\xymatrix{\bop_{i \geq 0} \textrm{H}_{\supp(\ff)}^{i}(X,\Omega _{X}^{\,i})\ar[r]&\bop_{i \geq 0} \textrm{H}\be^{i}(X,\Omega _{X}^{\,i})}
\] maps $\textrm{ch}(\ff)$ to the usual Chern character of $\ff$ in Hodge cohomology, which is obtained by taking the trace of the exponential of the Atiyah class of the tangent bundle $TX$\footnote{This property has been proved in \cite{Ca2} for algebraic varieties using different definitions of the HKR isomorphism and of the Hochschild class. In Kashiwara's setting, this is straightforward.}\!\!\!.
\par
\medskip
The Chern and Euler classes satisfy the same functoriality properties as the Hochschild and co-Hochschild classes, namely for every
 holomorphic map $f$ from $X$ to $Y$:
 \par \smallskip
\hspace{0.4 cm}-- for every $\g$ in $\textrm{D}_{\textrm{coh}}^{\,\textrm{b}}(Y)$,
$\ch(f^{*}\be\g)=f^{*}\ch(\g)$,
\par\smallskip
\hspace{0.4 cm}-- for every $\ff$ in $\textrm{D}_{\textrm{coh}}^{\,\textrm{b}}(X)$ such that
$f$ is proper on $\supp(\ff)$,
\[
\textrm{eu}(Rf_{!}\he\,\ff)=f_{!}\he\,\textrm{eu}(\ff).
\]
Furthermore, for every $\ff$ in $\textrm{D}_{\textrm{coh}}^{\textrm{b}}(X)$,
$\textrm{eu}(\ff)=\ch(\ff)\,\textrm{eu}(\oo_{X}\he)$. Put\-ting together the previous identity with the functoriality of the Euler class with respect to direct images, Kashiwara obtained the following Gro\-then\-dieck-Riemann-Roch theorem:
\begin{theorem}\label{KaGRR}\cite{Ka}
 Let $\apl{f}{X}{Y}$ be a holomorphic map and $\ff$ be an element of $\emph{D}_{\emph{coh}}^{\emph{b}}(X)$ such that $f$ is proper on $\supp(\ff)$. Then the following identity holds in $\bop_{i \geq 0} \emph{H}^{i}_{f[\supp(\ff)]}(Y,\Omega _{Y}^{\,i})\! \!:$
\[
 \ch(Rf_{!}\he\,\ff)\,\emph{eu}(\oo_{Y}\he)=f_{!}\he\,\bigl[\ch(\ff)\,\emph{eu}(\oo_{X}\he)\bigr].
\]
\end{theorem}
Then Kashiwara stated the following conjecture (see \cite[\S 5.3.4]{KaSc}):
\par\medskip
\textbf{Conjecture.} \cite{Ka} \emph{For any complex manifold $X$, the class $\emph{eu}(\oo_{X}\he)$ is the Todd class of the tangent bundle $TX$.}
\par\medskip
This conjecture was related to another conjecture of Schapira and Schneiders comparing the Euler class of a $\mathscr{D}_{X}\he$-module $\mathfrak{m}$ and the Chern class of the associated $\oo_{X}\he$-module $\textrm{Gr}(\mathfrak{m})$ (see \cite{ScSc}, \cite{BNT}).
\par\medskip
The aim of this note is to give a simple proof of Kashiwara's conjecture:
\begin{theorem}\label{Main}
For any complex manifold $X$, $\emph{eu}(\oo_{X}\he)$ is the Todd class of $TX$\!.
\end{theorem}
In the algebraic setting, an analogous result is established in \cite{Ra} (see also \cite{Ma}).
\par\medskip
As a corollary of Theorem \ref{Main}, we obtain the Grothendieck-Riemann-Roch theorem in Hodge cohomology for abstract complex manifolds, which has been already proved by different methods in \cite{OBToTo}:
\begin{theorem}\label{Suite}
 Let $\apl{f}{X}{Y}$ be a holomorphic map between complex manifolds, and let $\ff$ be an element of
$\emph{D}_{\emph{coh}}^{\emph{b}}(X)$ such that $f$ is proper on $\supp(\ff)$. Then
\[
 \ch(Rf_{!}\he\,\ff)\td(Y)=f_{!}\he\,\bigl[\ch(\ff)\td(X)\bigr].
\]
in $\bop_{i \geq 0} \emph{H}^{i}_{f[\supp(\ff)]}(Y,\Omega _{Y}^{\,i})$.
\end{theorem}
However, the proof given here is simpler and more conceptual. Besides, we would like to emphasize that it is entirely self-contained and relies only on the results
appearing in Chapter 5 of \cite{KaSc}.
\par\bigskip
\textbf{Acknowledgement.} I want to thank Masaki Kashiwara and Pierre Schapira for communicating their preprint \cite{KaSc} to me. I also want to thank Pierre Schapira for useful conversations, and Joseph Lipman for interesting comments.
\section{Notations and basic results}\label{SecDeux}
\setcounter{footnote}{1}
We follow the notation of \cite[Ch.\!\! 5]{KaSc}.
\par\medskip
If $X$ is a complex manifold, we denote by
$\textrm{D}^{\textrm{b}}\be(X)$ the bounded derived category of sheaves of $\oo_{X}\he$-modules and by $\textrm{D}^{\textrm{b}}_{\textrm{coh}}(X)$ the full subcategory of $\textrm{D}^{\textrm{b}}\be(X)$ consisting of complexes with coherent cohomology.
\par \medskip
If $\smash{\apl{f}{X}{Y}}$ is a holomorphic map between complex manifolds, the four ope\-rations
$\smash{\apl{f\ee\be}{\textrm{D}^{\textrm{b}}\be(Y)}{\textrm{D}^{\textrm{b}}\be(X),}}$\! $\smash{\apl{Rf_{*}\he \,, \, Rf_{!}\he \, }{\textrm{D}^{\textrm{b}}\be(X)}{\textrm{D}^{ \textrm{b}}\be(Y)}}$\!\! and $\smash{\apl{f^{!}\be\,}{\textrm{D}^{\textrm{b}}\be(Y)}{\textrm{D}^{\textrm{b}}\be(X)}}$
are part of the formalism of Grothen\-dieck's six operations. Let us recall their definitions:
\par \smallskip
\hspace{0.4 cm}-- $f\ee\be$ is the left derived functor of the pullback functor by $f$\!, that is $\mathcal{G} \flcourte \mathcal{G} \otimes_{f^{-1}\oo_{Y}\he}\oo_{X}\he$,
\par \smallskip
\hspace{0.4 cm}-- $Rf_{*}\he$ is the right derived functor of the direct image functor $f_{*}\he$, it is the left adjoint to the functor $f^{*}$\!\!,
\par \smallskip
\hspace{0.4 cm}-- $Rf_{!}\he$ is the right derived functor of the proper direct image functor $f_{!}\he$,
\par \smallskip
\hspace{0.4 cm}-- $f\pe\be$ is the exceptional inverse image, it is the right adjoint to the functor $Rf_{!}\he$.
\par \bigskip
If $W$ is a closed complex submanifold of $Y$\!, the pullback morphism from $f\ee\be \Omega _{Y}^{\, i}[i]$ to $ \Omega _{X}^{\, i}[i]$ induces in cohomology a map
\[
\apllongue{f\ee\be}{\, \bop_{i \geq 0}\textrm{H}^{i}_{W}( Y,\Omega _{Y}^{\, i})}{\bop_{i \geq 0} \textrm{H}^{i}_{f^{-1}(W)}( X,\Omega ^{\, i}_{X}).}
\]
\par
If $Z$ is a closed complex submanifold of $X$ and if $f$ is proper on $Z$, the integration morphism
from $\Omega _{X}^{\, i+d_{X} }[i+d_{X}\he ]$ to $\Omega _{Y}^{\, i+d_{Y}}[i+d_{Y}\he]$
induces a Gysin morphism
\[
\apllongue{f_{!}\, \he}{\, \bop_{i \geq -d_X}\textrm{H}_{Z}^{i+d_{X} }(X,\Omega _{X}^{\, i+d_{X}})}
{\bop_{i \geq -d_{Y}}\textrm{H}_{f(Z)}^{i+d_{Y}}(Y,\Omega _{Y}^{\, i+d_{Y}}).}
\]
\par \bigskip
Let X be a complex manifold, $\omega_{X}\he$ be the holomorphic dualizing complex of $X$ and $\delta_{X}\he$ be the diagonal injection. If $\ff$ belongs to $\textrm{D}_{\textrm{coh}}^{\,\textrm{b}}(X)$, we define the ordinary dual (resp.\ Verdier dual) of $\ff$ by the usual formula $D'\ff=\mathcal{RH}om_{\oo_{X}\he}\he(\ff, \oo_{X}\he)$ (resp.\ $D\ff=\mathcal{RH}om_{\oo_{X}\he}\he(\ff, \omega_{X}\he)$).
 \par \bigskip
The Hochschild and co-Hochschild classes of $\ff$, denoted by $\textrm{hh}_X\he(\ff)$ and $\textrm{thh}_X\he(\ff)$, lie in $\textrm{H}_{\supp(\ff)}^{0}(X,\delta _{X}^{*}\delta _{X*}\he\oo_{X}\he)$ and $\textrm{H}_{\supp(\ff)}^{0}(X,\delta _{X}^{!}\delta _{X!}\he\,\omega _{X}\he)$ respectively.
They are constructed by the chains of maps
\[
\xymatrix@C=17pt@R=-2pt{
\textrm{hh}_{X}\he(\ff):\, \, \, \textrm{id}\ar[r]& \mathcal{RH}om(\ff, \ff)\ar[r]^-{\sim}& \delta_{X}^{*} (D^{'}\!\ff \boxtimes_{\oo} \ff)\ar[r]& \delta_{X}^{*} \delta_{X*}\he \, \oo_{X}\he\\
\textrm{thh}_X\he(\ff):\, \, \, \textrm{id}\ar[r]&\mathcal{RH}om(\ff, \ff)\ar[r]^-{\sim}& \delta_{X}^{!} (D\ff \boxtimes_{\oo} \ff)\ar[r]& \delta_{X}^{!} \delta_{X!}\he \,\omega_{X}\he
}
\]
where in both cases the last arrows are obtained from the derived trace maps
\[
\flgd{D'{}\ff \,\lltens{}_{\oo_{X}\he} \ff }{\oo_{X}\he}\quad\textrm{and}\quad
\flgd{D\ff \, \lltens{}_{\oo_{X}\he} \ff }{\omega_{X}\he}
\]
by adjunction.
\par \bigskip
If $\apl{f}{X}{Y}$ is a holomorphic map between complex manifolds, there are pullback and push-forward morphisms
\[
\apl{f^{*}}{f^{*}\delta_{Y}^{*} \delta_{Y*}\he \, \oo_Y\he}{\delta_{X}^{*} \delta_{X*}\he \, \oo_{X}\he}\ \textrm{and}\quad \apl{f_{!}\, }{\, Rf_{!}\, \delta_{X}^{!} \delta_{X!}\he \omega_{X}\he}{\delta_{Y}^{!} \delta_{Y!}\he \omega_{Y}\he.}
\]
Besides, there is a natural pairing
\begin{equation}\label{un}
\delta_{X}^{*} \delta_{X*}\he \, \oo_{X}\he \lltens{}_{\oo_{X}
\he}\he \, \delta_{X}^{!} \delta_{X!}\he \, \omega_{X}\he \flcourte \delta_{X}^{!} \delta_{X!}\he \, \omega_{X}\he
\end{equation}
given by the chain
\[
\delta_{X}^{*} \delta_{X*}\he \, \oo_{X}\he \lltens{}_{\oo_{X}
\he}\he \, \delta_{X}^{!} \delta_{X!}\he \, \omega_{X}\he \simeq
\delta_{X}^{!} (\delta_{X*}\he \, \oo_{X}\he \lltens{}_{\oo_{X \times X}
\he}\he \, \delta_{X!}\he \, \omega_{X}\he) \flcourte
\delta_{X}^{!} \delta_{X!}\he \, \omega_{X}\he.
\]
\begin{theorem}\cite{Ka} \label{boss1}
\begin{enumerate}
\item [(i)] For all elements $\ff$ and $\g$ in $\textrm{D}_{\emph{coh}}^{\,\emph{b}}(X)$ and $\textrm{D}_{\emph{coh}}^{\,\emph{b}}(Y)$ respectively,
$\emph{hh}_{X}\he (f^{*}\g)=f^{*}\, \emph{hh}_{Y}\he(\g)$ and
$f_{!}\, \emph{thh}_{X}\he(\ff)=\emph{thh}_{Y}\he (Rf_{!}\ff)$.
\par \smallskip
\item [(ii)] Via the pairing \emph{(}\ref{un}\emph{)}, for every $\ff$ in
$\textrm{D}_{\emph{coh}}^{\,\emph{b}}(X)$,
\[
\emph{hh}_{X}\he(\ff)\, \emph{thh}(\oo_{X}\he)=\emph{thh}_{X}\he (\ff).
\]
\end{enumerate}

\end{theorem}
The Hochschild and co-Hochschild classes are translated into Hodge co\-ho\-mology classes by Kashiwara's analytic Hochschild-Kostant-Rosen\-berg isomorphisms\footnote{For a detailed account of the HKR isomorphisms, we refer to the introduction of \cite{Gr} and to the references therein.}
\begin{equation} \label{deux}
 \delta _{X}^{*}\delta_{X*}\he\oo_{X}\he\simeq\xymatrix{\bop\limits_{i \geq 0}} \Omega_{X}^{\, i}[i]\qquad \textrm{and}\qquad\delta _{X}^{!}\delta_{X!}\he\,\omega _{X}\he\simeq\xymatrix{\bop\limits_{i \geq 0}} \Omega_{X}^{\, i}[i],
\end{equation}
the resulting classes are called Chern and Euler classes. \!\!If $\ff$ is an ele\-ment of $\textrm{D}_{\textrm{coh}}^{\textrm{b}}(X)$, then $\textrm{ch}(\ff)$ and $\textrm{eu}(\ff)$ lie in $\bop_{i \geq 0} \textrm{H}^{i}_{\supp(\ff)}(X,\Omega _{X}^{\,i}).$
\par \medskip
The first HKR isomorphism commutes with pullback and the second one with push forward. Besides, the pairing (\ref{un}) between $\delta_{X}^{*} \delta_{X*}\he \, \oo_{X}\he $ and $\delta_{X}^{!} \delta_{X!}\he \, \omega_{X}\he$ is exactly the cup-product on holomorphic differential forms after applying the HKR isomorphisms (\ref{deux}).

\begin{theorem}\label{boss2}\cite{Ka}
\begin{enumerate}
\item[(i)]  For every $\ff$ in $\textrm{D}_{\emph{coh}}^{\,\emph{b}}(X)$, $\emph{ch}(\ff)$ is the usual Chern character obtained by the Atiyah exact sequence.
\par \smallskip
\item [(ii)] For all elements $\ff$ and $\g$ in $\textrm{D}_{\emph{coh}}^{\,\emph{b}}(X)$ and $\textrm{D}_{\emph{coh}}^{\,\emph{b}}(X)$ respectively,
$\emph{ch}(f^{*}\g)=f^{*}\, \emph{ch}(\g)$ and $f_{!}\, \emph{eu}(\ff)=\emph{eu} (Rf_{!}\ff)$.
\par\smallskip
\item [(iii)] For every $\ff$ in $\textrm{D}_{\emph{coh}}^{\,\emph{b}}(X)$, $\emph{eu(\ff)}=\emph{ch}(\ff)\,\emph{eu}(\oo_X)$.
\end{enumerate}
\end{theorem}
For the proofs of Theorems \ref{boss1} and \ref{boss2}, we refer to \cite[Ch.\ 5]{KaSc}.
\par \medskip
For any complex manifold $X$, we denote by $\textrm{td}(X)$ the Todd class of the holomorphic tangent bundle $TX$ in $\bop_{i \geq 0}\textrm{H}^{i}(X,\Omega _{X}^{\, i})$.

\section{Proof of Theorem \ref{Main}}
We proceed in several steps.
\begin{proposition}\label{LemUn}
 Let $Y$ and $Z$ be complex manifolds such that $Z$ is a closed complex submanifold of $Y$, and let $i_Z \he$ be the corresponding inclusion. Then, for every coherent sheaf $\ff$ on $Z$, we have
\[
i_{Z!}\he\,[\ch(\ff)\td(Z)]=\ch(i_{Z*}\he\ff)\td(Y)
\]
in $\bop_{i \geq 0} \emph{H}^{i}_{Z}( Y,\Omega _{Y}^{\,i})$.
\end{proposition}
\begin{proof}
This is proved in the classical way using the deformation to the normal cone as in \cite[\S 15.2]{Ful}, except that we use local cohomology.
For the sake of completeness, we provide a detailed proof.
\par\bigskip
We start by a particular case:
\par \smallskip
\hspace{0.4 cm}-- $\mathcal{N}$ is a holomorphic vector bundle on $Z$, and $Y=\mathbb{P}\,(\mathcal{N} \oplus \mathcal{O}_Z)$,
\par \smallskip
\hspace{0.4 cm}-- $Z$ embeds in $Y$ by identifying $Z$ with the zero section of $\mathcal{N}$\!.
\par \medskip
Let $d$ be the rank of $\mathcal{N}$, $\pi$ be the projection of the projective bundle $\mathbb{P}\,(\mathcal{N} \oplus \mathcal{O}_Z)$
and $\mathcal{Q}$ be the universal quotient bundle on $Y$; $\mathcal{Q}$ is the quotient of $\pi^*(\mathcal{N} \oplus \mathcal{O}_Z)$ by the tautological line bundle $\smash[b]{\mathcal{O}_{\mathcal{N} \oplus \, \mathcal{O}_Z}\he(-1)}$.
Then $\mathcal{Q}$ has a canonical holomorphic section $s^{\vphantom{{(}}}$ which is obtained by the composition
\[
s \colon \mathcal{O}_Y \he \simeq \pi^{*}\be \mathcal{O}_Z \he \fl  \pi^{*}\be (\mathcal{N} \oplus \mathcal{O}_Z) \fl \mathcal{Q}.
\]
This section vanishes transversally along its zero locus, which is exactly $Z$.
Therefore,
we have a natural locally free resolution of $i_{Z!}\he\mathcal{O}_Z$ given by the Koszul complex associated with the pair $(\mathcal{Q}^{*} \be, s^{*} \be)$:
\[
\xymatrix@C=17pt{0\ar[r]&\wedge^{d}\be \mathcal{Q}^{*} \be\ar[r]&\wedge^{d-1}\be \mathcal{Q}^{*}\ar[r]&\cdots\ar[r]&\oo_{Y}\he\ar[r]&i_{Z!}\he\oo_{Z}\he \ar[r]&0.}
\]
This gives the equality
\[
\textrm{ch}(i_{Z!}\he\oo_Z\he)=\sum_{k=0}^{d} (-1)^k \textrm{ch}(\wedge^{k} \mathcal{Q}^{*} \be)=\textrm{c}_d\he(\mathcal{Q}) \, \textrm{td}(\mathcal{Q})^{-1} \be
\]
in $\bop_{i \geq 0} \textrm{H}^i\be(Y, \Omega_Y^i)$, where $\textrm{c}_d\he(\mathcal{Q})$ denotes the $d$-th Chern class of $\mathcal{Q}$ (for the last equality, see \cite[\S \ 3.2.5]{Ful}).
Since $\textrm{c}_d\he(\mathcal{Q})$ is the image of the constant class $1$ by $i_{Z!}\he$ and since $i_Z^{*}\mathcal{Q}=\mathcal{N}$\!, we get
\[
\textrm{ch}(i_{Z!}\he\oo_Z\he)=i_{Z!}\he(i_Z^{*} \textrm{td}(\mathcal{Q})^{-1} \be)=i_{Z!}\he(\textrm{td}(\mathcal{N})^{-1} \be).
\]
\par
For any coherent sheaf $\mathcal{F}$ on $Z$, we have $i_{Z!}\he \mathcal{F}=i_{Z!}\he\mathcal{O}_Z \he \, \lltens{}_{\mathcal{O}_Y \he}\he \pi^{*} \be \mathcal{F}$
so that we obtain by the projection formula
\begin{equation}\label{trois}
\textrm{ch}(i_{Z!}\he\mathcal{F})=i_{Z!}\he (\textrm{ch}(\mathcal{F}) \textrm{td}(\mathcal{N})^{-1} \be)
\end{equation}
in $\bop_{i \geq 0} \textrm{H}^i(Y, \Omega_Y^i)$. Remark now that by Theorem \ref{boss2} (ii) and (iii), we have
\[
\ch({i}_{Z!}\he\, \ff)={i}_{Z!}\he( \ch(\ff) \textrm{eu}(\oo_{Z}\he) \,{i}_{Z}^{*}\, \textrm{eu}(\oo_{Y}\he)^{-1})
\]
\par \smallskip
\noindent in $\bop_{i \geq 0} \textrm{H}_{Z}^{i}(Y,\Omega _{Y}^{\, i})$. This proves that $\ch({i}_{Z!}\he\, \ff)$ lies in the image of
\[
\apllongue{{i}_{Z!}\he}{\bop_{i \geq 0} \textrm{H}^{i}\be( Z,\Omega _{Z}^{\, i})}{ \bop_{i \geq 0} \textrm{H}_{Z}^{i+d}(Y,\Omega _{Y}^{\, i+d}).}
\]
Let us denote this image by $W$\!.
The map
\[
\apllongue{\iota}{W}{\bop_{i \geq 0} \textrm{H}^{i+d}(Y,\Omega _{Y}^{\, i+d})}
\]
obtained by forgetting the support is injective. Indeed, for every class $i_{Z!}\he \alpha$ in $W$\!, $\pi_{!}\he [\iota(i_{Z!}\he \alpha)]=\alpha$.
This implies that (\ref{trois}) holds in $\bop_{i \geq 0} \textrm{H}_{Z}^{i}(Y,\Omega _{Y}^{\, i})$.
\par\bigskip
We now turn to the general case, using deformation to the normal cone. Let us introduce some notation:
\par \smallskip
\hspace{0.4 cm}-- $M$ is the blowup of $Z\tim\{0\}$ in $Y\tim\P^{1}$\!, $\sigma$ is the blowup map and $q=\pr_{1}\he\!\circ\, \sigma$\!,
\par \smallskip
\hspace{0.4 cm}-- for any divisor $D$ on $M$, $[D]$ denotes its cohomological cycle class in $\textrm{H}^{1}(M,\Omega _{M}^{1})$,
\par \smallskip
\hspace{0.4 cm}-- $E=\P( N_{Z/Y}\he\oplus \oo_{Z}\he)$ is the exceptional divisor of the blowup and $\ti{Y}$ is the strict transform of $Y\tim\{0\}$ in $M$\!,
\par \smallskip
\hspace{0.4 cm}-- $\mu$ is the embedding of $Z$ in $E$, where $Z$ is identified with the zero section of $N_{Z/Y}\he$,
\par \smallskip
\hspace{0.4 cm}-- $F$ is the embedding of (the strict transform of) $Z\tim\P^{1}$ in $M$ and for any $t$ in $\P^1$\!, $j_t\he$ is the embedding of $M_t\he$ in $M$\!,
\par \smallskip
\hspace{0.4 cm}-- $k$ is the embedding of $E$ in $M$\!.
\par \medskip

Then $M$ is flat over $\P^{1}$\!, $M_0$ is a Cartier divisor with two smooth components $E$ and $\ti{Y}$ intersecting transversally along $\mathbb{P}(N_{Z/Y}\he)$, and $M_t\he$ is isomorphic to $Y$ if $t$ is nonzero.

\par\medskip
Let $\g=F_{!}\he (\pr_{1}^{*}\ff )$. Since $M$ is flat over $\P^{1}$\!, for any $t$ in $\P^1 \setminus \{0\}$,
\[
j_{t}\ee\g=i_{Z!}\he\ff \quad  \textrm{and} \quad k\ee\be\g=\mu_{!}\he\ff\!.
\]
If $\ch(\g)$ is the Chern character of $\g$ in $\bop_{i \geq 0} \textrm{H}^{i}_{Z\tim\P^{1}}( M,\Omega _{M}^{\, i})$, using the identity
(\ref{trois}) in $\bop_{i \geq 0} \textrm{H}^{i}_{Z}(E, \Omega _{E}^{\, i})$, we get
\par
\begin{align}
 j_{t!}\he\ch(i_{Z!}\he\ff)&=j_{t!}\he\,  j_{t}\ee\ch(\g)=\ch(\g)\, [M_{t}\he]\notag\\
&=\ch(\g)\, [M_{0}]=\ch(\g)\, [E]+ \ch(\g)\,[\ti{Y}]\notag\\
&=\ch(\g)\, [E]=k_{!}\he\, k\ee\be\ch(\g)=k_{!}\he\ch(\mu_{!}\he\ff)\notag\\
&=k_{!}\he\, \mu_{!}\he(\ch(\ff) \td(N_{Z/E}\he)^{-1})\notag\\
&=k_{!}\he\, \mu_{!}\he(\ch(\ff)\td(N_{Z/X}\he)^{-1})\notag
\end{align}
in $\bop_{i \geq 0} \textrm{H}^{i}_{Z\tim\P^{1}}( M,\Omega _{M}^{\, i})$.
\par \medskip
The map $q$ is proper on $Z\tim\P^{1}$, $q\circ j_{t}\he=\id$ and $q\circ k\circ \mu=i_{Z}\he$. Applying $q_{!}\he$, we get
\[
\ch( i_{Z!}\he\ff)=i_{Z!}\he(\ch(\ff) \td(N_{Z/X}\he)^{-1})
\]
in $\bop_{i \geq 0} \textrm{H}^{i}_{Z}(Y,\Omega _{Y}^{\, i})$.
\end{proof}
\begin{definition}
 For any complex manifold $X$, let $\alpha (X)$ be the cohomology class in $\bop_{i \geq 0}\textrm{H}^{i}(X,\Omega ^{i}_{X})$ defined by $\alpha(X)=\textrm{eu}(\oo_{X}\he)\td(X)^{-1}$\!\!.
\end{definition}
\begin{lemma}\label{LemDeux}
Let $Y$ and $Z$ be complex manifolds such that $Z$ is a closed complex submanifold of $Y$\!, and let $i_Z\he$ be the corresponding injec\-tion.
Assume that there exists a holomorphic retraction $R$ of $i_Z$. Then we have
$
\alpha (Z)=i_{Z}^{*}\,\alpha (Y).
$

\end{lemma}
\begin{proof}
By Theorem \ref{boss2} (ii), $\textrm{eu}(i_{Z*}\he\oo_{Z}\he)=i_{Z!}\he\,\textrm{eu}(\oo_{Z}\he)$. By Proposition \ref{LemUn} and Theorem \ref{boss2} (iii),
\[
 \textrm{eu}(i_{Z*}\he\oo_{Z}\he)=\ch(i_{Z*}\he\oo_{Z}\he)\textrm{eu}(\oo_{Y}\he)=
(i_{Z!}\he\td(Z))\td(Y)^{-1}\textrm{eu}(\oo_{Y}\he),
\]
so that we obtain in $\bop_{i \geq 0} \textrm{H}^{i}_{Z}( Y,\Omega _{Y}^{\, i})$ the formula
\[
 i_{Z!}\he\bigl[
\textrm{eu}(\oo_{Z}\he)-\td(Z)\,i_{Z}^{*}(\textrm{eu}(\oo_{Y}\he)\td(Y)^{-1})
\bigr]=0.
\]
\par \smallskip
Since $R$ is proper on $Z$, we can apply $R_{!}\he$ and we get the result.
\par
\end{proof}
\begin{lemma}\label{LemTrois}
 The class $\alpha(X)$ satisfies $\alpha (X)^2=\alpha (X)$.
\end{lemma}
\begin{proof}
We apply Lemma \ref{LemDeux} with $Z=X$ and $Y=X\times X$, where $X$ is diagonally embedded in $X\times X$.
Then $\alpha (X)=i_{\Delta  }^{*}\,\alpha (X\times X)$. The Euler class commutes with external products so that
\[
\textrm{eu}(\oo_{X\times X}\he)=\textrm{eu}(\oo_{X}\he)\boxtimes\, \textrm{eu}(\oo_{X}\he)
\]
\par \smallskip
Thus $\alpha (X\times X)=\alpha (X)\boxtimes\,\alpha (X)$ and we obtain
\[
\alpha (X)=i_{\Delta  }^{*}[\, \alpha (X)\boxtimes\,\alpha (X)\, ]=\alpha (X)^2\!.
\]
\end{proof}
\textit{Proof of Theorem \ref{Main}.}
There is a natural isomorphism $\phi$ in $\textrm{D}_{\textrm{coh}}^{\,\textrm{b}}(X)$ between $\delta ^{*}\be\delta _{*}\he\oo_{X}\he$ and $\delta ^{!}\be\delta _{!}\he\omega _{X}\he$ given by the chain
\[
\delta ^{*}\be\delta _{*}\he\oo_{X}\he \simeq \mathcal{O}_X \he \, \lltens{}_{\mathcal{O}_X \he} \he \delta ^{*}\be\delta _{*}\he\oo_{X}\he
\simeq \delta ^{!}(\omega_X \he \boxtimes\, \mathcal{O}_X) \, \lltens{}_{\mathcal{O}_X \he} \he \delta ^{*}\be\delta _{*}\he\oo_{X}\he
 \simeq \delta ^{!}\be\delta _{!}\he\omega _{X}\he.
\]
Besides, after applying the two HKR isomorphisms (\ref{deux}), $\phi$ is given by derived cup-product with the Euler class of $\mathcal{O}_X \he$ (see \cite{KaSc}). Therefore, the class $\textrm{eu}(\mathcal{O}_X \he)$ is invertible in the Hodge cohomology ring of $X$, and so is $\alpha(X)$. Lemma \ref{LemTrois} implies that $\alpha(X)=1$.
\par\smallskip  \hfill q.e.d.\!

\end{document}